\documentclass[11pt]{amsart}
\usepackage{amsmath,amscd}
\usepackage{amsfonts}
\usepackage{amssymb}

\newcommand{\FF}{{\mathbb F}}

\newcommand{\cE}{{\mathcal E}}

\newcommand{\bP}{{\bf P}}
\newcommand{\bG}{{\bf G}}

\newcommand{\bS}{{\bf S}}
\newcommand{\bT}{{\bf T}}
\newcommand{\bB}{{\bf B}}

\newcommand{\bL}{{\bf L}}
\newcommand{\bN}{{\bf N}}

\newcommand{\Nr}{\mathrm{N}}

\newcommand{\Ind}{\mathrm{Ind}}

\newtheorem{thm}{Theorem}[section]
\newtheorem{lemma}{Lemma}[section]

\def\adots{\mathinner{\mkern2mu\raise0pt\hbox{.}  
\mkern2mu\raise4pt\hbox{.}\mkern1mu
\raise7pt\vbox{\kern7pt\hbox{.}}\mkern1mu}}

\numberwithin{equation}{section}

\begin{document}

\bibliographystyle{ieeetr}

\title[Jordan decomposition and real-valued characters]
{Jordan decomposition and real-valued characters of finite reductive groups with connected center}
\author{Bhama Srinivasan}
  \address{Department of Mathematics, Statistics, and Computer Science (MC 249)\\
           University of Illinois at Chicago\\
           851 South Morgan Street\\
           Chicago, IL  60680-7045}
  \email{srinivas@uic.edu}
\author{C. Ryan Vinroot}
  \address{Department of Mathematics\\
           College of William and Mary\\
           P. O. Box 8795\\
           Williamsburg, VA  23187-8795}
   \email{vinroot@math.wm.edu}

\begin{abstract}  Let $\bG$ be a connected reductive group with connected center defined over $\FF_q$, with Frobenius morphism $F$.  We parameterize all of the real-valued irreducible complex characters of $\bG^F$ using the Jordan decomposition of characters.
\end{abstract}

\maketitle

\section{Introduction} \label{Intro}

Let $\bG$ be a connected reductive group with connected center defined over $\FF_q$, and $F$ a Frobenius morphism of $\bG$.  Our main result is a parametrization of the real-valued irreducible complex characters of the finite group $\bG^F$ in terms of the Jordan decomposition of characters.  We hope that the parametrization of the real-valued characters of finite reductive groups will be very useful in the real-valued character theory of finite groups in general.  For example, the results of Guralnick, Navarro, and Tiep \cite{GuNaTi11} on the sizes of real classes and degrees of real characters depends on the examination of real characters of many finite reductive groups.

After establishing notation and preliminaries in Section \ref{Prelim}, we discuss the Jordan decomposition of characters in Section \ref{Jord}.  In particular, we state a result of Digne and Michel in Theorem \ref{UniqueJord}, which states that when the center of $\bG$ is connected, the Jordan decomposition map may be uniquely described by a certain list of properties.  We use this result in a crucial way in the proof of our main result, Theorem \ref{MainThm}.  That is, given the image of some character $\chi$ under the Jordan decomposition map, then we find the image of $\bar{\chi}$ under the map, by showing the proposed image satisfies the list of properties from the result of Digne and Michel.  We may state our main result as follows.  Let $\chi$ be an irreducible complex character of $\bG^F$, and let $\chi$ correspond to the pair $(s_0, \nu)$ via the Jordan decomposition, where $s_0$ is a semisimple element of the dual group $\bG^{*F^*}$, and $\nu$ is a unipotent character of $C_{\bG^*}(s_0)^{F^*}$.  Then $\chi$ is real-valued if and only if the pair $(s_0, \nu)$ is $\bG^{*F^*}$-conjugate to the pair $(s_0^{-1}, \bar{\nu})$.\\
\\
\noindent {\bf Acknowledgements. }  The second-named author was supported in part by a grant from the Simons Foundation.

\section{Preliminaries} \label{Prelim}

Throughout this section, we let $\bG$ be a connected reductive group defined over a finite field $\FF_q$, and $F: \bG \rightarrow \bG$ the corresponding Frobenius morphism.  For any $F$-stable subgroup $\bG_1$ of $\bG$, we let $\bG_1^F$ denote the group of $F$-fixed elements of $\bG_1$.  For an element $g \in \bG$, we denote ${^g \bG_1} = g \bG_1 g^{-1}$.  We begin with the following. 

\begin{lemma} \label{Norm} Let $\bG_1$ be a connected reductive $F$-stable subgroup of $\bG$, and $\bS$ an $F$-stable maximal torus of $\bG_1$ and $\bG$.  Let $\bN_1 = \Nr_{\bG}(\bG_1)$ and $\bN = \Nr_{\bG}(\bS)$.  Then $\bN_1 = \bG_1 (\bN \cap \bN_1)$.  

If $\bS \subset \bB$ where $\bB$ is an $F$-stable Borel subgroup of $\bG_1$, then $\bN_1^{F} = \bG_1^{F} (\bN^{F} \cap \bN_1^{F})$.
\end{lemma}
\begin{proof} Let $g \in \bN_1$.  Since ${^g \bG_1} = \bG_1$ and $\bS \subseteq \bG_1$, we have ${^g \bS} \subseteq \bG_1$.  Any two maximal tori of $\bG_1$ are conjugate, and so ${^g \bS} = {^x \bS}$ for some $x \in \bG_1$.  Then $x^{-1} g = w \in \bN$, and $g = xw$, where $x \in \bG_1$ and $w \in \bN \cap \bN_1$, which gives the first statement.

Now suppose $g \in \bN_1^{F}$.  Then ${^g \bS}$ and ${^g \bB}$ are $F$-stable.  By the Lang-Steinberg Theorem, the pairs $(\bS, \bB)$ and $({^g \bS}, {^g \bB})$ are conjugate by some $x \in \bG^{F}$.  Then ${^g \bS} = {^x \bS}$, and $x^{-1} g \in \bN^{F}$, so $g = xw$ with $x \in \bG_1^{F}$ and $w \in \bN^{F} \cap \bN_1^{F}$, giving the second result.
\end{proof}

Fix a pair $(\bT, \bB)$, where $\bT$ is a maximally split $F$-stable torus of $\bG$, contained in an $F$-stable Borel subgroup $\bB$ of $\bG$.  Through the root system associated with $\bT$, we define the dual reductive group $\bG^*$ with dual Frobenius morphism $F^*$, and with $F^*$-stable maximal torus $\bT^*$ dual to $\bT$, contained in the $F^*$-stable Borel $\bB^*$ of $\bG^*$.  Define the Weyl group $W = \Nr_{\bG}(\bT)/\bT$, and the dual Weyl group $W^* = \Nr_{\bG^*}(\bT^*)/\bT^*$.  There is a natural isomorphism $\delta: W \rightarrow W^*$ \cite[Sec. 4.2]{Ca85}, and a corresponding anti-isomorphism, $w \mapsto w^* = \delta(w)^{-1}$.  The isomorphism $\delta$ restricts to an isomorphism between $W^F = \Nr_{\bG}(\bT)^F/\bT^F$ and $(W^*)^{F^*}$ \cite[Sec. 4.4]{Ca85}.  We let $\ell$ denote the standard length function on these Weyl groups.

We recall that the $\bG^F$-conjugacy classes of $F$-stable maximal tori in $\bG$ may be classified by $F$-conjugacy classes of $W$ as follows \cite[Sec. 8.2]{CaEn04}.  If $\bT'$ is an $F$-stable torus in $\bG$, then $\bT' = {^g \bT}$ for some $g \in \bG$.  Then $g^{-1} F(g) \in \Nr_{\bG}(\bT)$ and $w = g^{-1} F(g) \bT \in W$.  The $\bG^F$-conjugacy class of $\bT'$ then corresponds to the $W\langle F \rangle$-conjugacy class of $wF$, which corresponds to the $F$-conjugacy class of $w$ in $W$.  Then we have $\bT'^F = g (\bT^{wF})g^{-1}$.  We say that $\bT'$ is an $F$-stable torus of $\bG$ of type $w$ (noting that the reference torus $\bT$ is fixed).  Since $\bT'^F$ and $\bT^{wF}$ are isomorphic, we work with $\bT^{wF}$ instead of $\bT'^F$.

Similar to the case of tori, the $\bG^F$-conjugacy classes of $F$-stable Levi subgroups are classified as follows.  Let $\bL$ be a Levi subgroup of a standard parabolic $\bP$, and given $w \in W$, let $\dot{w}$ denote an element in $\Nr_{\bG}(\bT)$ which reduces to $w$ in $W$.  Then any Levi subgroup of $\bG^F$ is isomorphic to $\bL^{\dot{w}F}$ for some $w \in W$, and we work with $\bL^{\dot{w}F}$ instead of the Levi subgroup of $\bG^F$.  For precise statements, see \cite[Sec. 8.2]{CaEn04}, \cite[Prop. 4.3]{DiMi90}, or \cite[Prop. 26.2]{MaTe11}.

If $\bT'$ is an $F$-stable torus of $\bG$ which is type $w$, then the $F^*$-stable maximal torus of type $F^*(w^*)$ in $\bG^*$ (with respect to $\bT^*$) is exactly the dual torus $(\bT')^*$ of $\bT'$.  It follows that the finite tori $\bT^{wF}$ and $\bT^{*(wF)^*}$ are in duality, and there is an isomorphism, which we fix as in \cite[Sec. 8.2]{CaEn04}, between $\bT^{*F^*}$ and the group of characters $\hat{\bT}^F$ of $\bT^F$,
\begin{align*}
\bT^{*F^*} & \longleftrightarrow   \hat{\bT}^F \\
s & \longmapsto \theta = \hat{s}.
\end{align*}
Since $\bT^{wF}$ is in duality with $\bT^{*(wF)^*}$, then we may replace $F$ with $wF$, and $F^*$ with $(wF)^*$ in the correspondence above.   In particular, if $s \in \bT^{*(wF)^*}$ for some $w \in W$, then we denote by $\hat{s}$ the corresponding character in $\hat{\bT}^{wF}$.

Consider any semisimple element $s_0 \in \bG^{*F^*}$.  Then $s_0$ is contained in an $F^*$-stable maximal torus of $\bG^*$, and as above, we have $s_0 \in g(\bT^{*(wF)^*})g^{-1}$ for some $w \in W$ and $g \in \bG^*$, which we fix with respect to $s_0$.  We correspond to $s_0$ the element $s = g^{-1} s_0  g\in \bT^*$, where $s$ is $(wF)^*$-fixed.  Given any element $s \in \bT^*$, define $W_F(s)$ as
$$ W_F(s) = \{ w \in W \, \mid \, {^{(wF)^*} s} = s \}.$$
Then, the semisimple elements in $\bG^{*F^*}$ correspond to elements $s \in \bT^*$ such that $W_F(s)$ is nonempty.  Given such an $s \in \bT^*$, consider $C_{\bG^*}(s)$ and its Weyl group $W^*(s)$ relative to $\bT^*$, and define $W(s)$ to be the collection of elements $w \in W$ such that $w^* \in W^*(s)$.  Then, as in \cite[Section 2]{DiMi90}, we may write 
$$ W_F(s) = w_1 W(s),$$
where $w_1 \in W_F(s)$ is such that $\bT^{*(w_1 F)^*}$ is the maximally split torus inside of $C_{\bG^*}(s)^{(\dot{w}_1 F)^*}$.  The maximal tori in $C_{\bG^*}(s)^{(\dot{w}_1 F)^*}$ are then each isomorphic to a torus of the form $\bT^{*(wF)^*}$, for $w \in W_F(s)$, by the same classification of maximal tori which we applied to $\bG^{*F^*}$.  We now give an application of Lemma \ref{Norm}.

\begin{lemma} \label{NormFrob}  Suppose $s, t \in \bT^{*(w_1F)^*}$, so $s, t \in C_{\bG^*}(s)^{(\dot{w}_1 F)^*}$, and suppose $C_{\bG^*}(s)$ is connected.  If $h \in \bG^{*(\dot{w}_1 F)^*}$ such that $h s h^{-1} = t$, then there is an element $\dot{v} \in \Nr_{\bG^*}(\bT^*)^{(\dot{w}_1 F)^*}$ such that $\dot{v} s \dot{v}^{-1} = t$.
\end{lemma} 
\begin{proof}  We apply Lemma \ref{Norm} to the group $\bG^*$ with Frobenius morphism $(\dot{w}_1 F)^*$,  and with $\bG_1 = C_{\bG^*}(s)$, and $\bS = \bT^*$, so that $\bN^{F}$ in Lemma \ref{Norm} is $\Nr_{\bG^*} (\bT^*)^{(\dot{w}_1 F)^*}$.  Then, if $h \in \bG^{*(\dot{w}_1 F)^*}$ such that $h s h^{-1} = s^{-1}$, it follows that $h \in \Nr_{\bG^*}(C_{\bG^*}(s))^{(\dot{w}_1 F)^*} $.  It then follows from Lemma \ref{Norm} that there exists an element $\dot{v} \in \Nr_{\bG^*}(\bT^*)^{(\dot{w}_1 F)^*}$ such that $\dot{v} s \dot{v}^{-1} = t$.
\end{proof}

We remark that given $s, t \in \bT^{*(w_1F)^*}$ as in Lemma \ref{NormFrob}, we have $s$ and $t$ are conjugate in $\bG^{*(\dot{w}_1 F)^*}$ if and only if the corresponding semisimple elements in $\bG^{*F^*}$ are $\bG^{*F^*}$-conjugate, which follows from a direct calculation.

\section{Jordan decomposition of characters}  \label{Jord}

We recall the description of irreducible characters of a finite reductive group.  The results we state are presented in the texts \cite{CaEn04, Ca85, dmbook}, and were originally proved by Deligne and Lusztig \cite{dellusz} and Lusztig \cite{Lu78, Lu84, Lu88}.

We have $\bG$ a connected reductive group defined over a finite field $\FF_q$, and $F$ a Frobenius morphism of $\bG$, and we continue with the established notation in Section \ref{Prelim}.  From here, we assume that the center of $\bG$ is connected.  If $\theta$ is any character of $\bT^{wF}$, recall that there is generalized character of $\bG^F$, which we denote by $R_{\bT^{wF}}^{\bG^F} (\theta)$, as constructed by Deligne and Lusztig \cite{dellusz}.  Given any $s \in \bT^*$ with $W_F(s)$ nonempty, we let $\cE(\bG^F, s)$ denote the \emph{(geometric) Lusztig series} corresponding to $s$, which is the set of irreducible characters $\chi$ of $\bG^F$ such that, for some $w \in W_F(s)$, we have
$$ \langle \chi, R_{\bT^{wF}}^{\bG^F} (\hat{s}) \rangle \neq 0,$$
where $\langle \cdot, \cdot \rangle$ is the standard inner product on class functions.  Then the Lusztig series $\cE(\bG^F, s)$ partition the set of all irreducible characters of $\bG^F$, and two Lusztig series $\cE(\bG^F, s)$ and $\cE(\bG^F, t)$ coincide if and only the semisimple elements of $\bG^{*F^*}$ which correspond to $s, t \in \bT^*$ are $\bG^{*F^*}$-conjugate.  In particular, if $s_0$ is a semisimple element of $\bG^{*F^*}$ which corresponds to $s$, then we also denote the Lusztig series by $\cE(\bG, s_0)$.  The \emph{unipotent} characters of $\bG^F$ are those irreducible characters which are elements of $\cE(\bG^F, 1)$.  One can further define \emph{rational} Lusztig series.  We refer to \cite[p. 127]{CaEn04} and recall that the geometric and rational Lusztig series coincide when $\bG$ has connected center, or in the case of unipotent characters even when the center is not necessarily connected.

 Lusztig \cite{Lu84} constructed an explicit bijection, which we call the {\it Jordan decomposition} map,
\begin{equation} \label{JD1}
 J_s: \cE(\bG, s) \longrightarrow \cE(C_{\bG^*}(s)^{(\dot{w}_1 F)^*}, 1),
\end{equation}
with the property that, for any $w \in W_F(s)$,
\begin{equation} \label{inner}
\langle \chi, R_{\bT^{wF}}^{\bG^F}(\hat{s}) \rangle = \langle J_s(\chi), (-1)^{\ell(w_1)} R_{\bT^{*(wF)^*}}^{C_{\bG^*}(s)^{(\dot{w}_1 F)^*}} ({\bf 1}) \rangle.
\end{equation}
If we denote our Lusztig series in terms of the semisimple element $s_0 \in \bG^{*F^*}$, then we write the Jordan decomposition map as
\begin{equation} \label{JD2}
J_{s_0}: \cE(\bG, s_0) \longrightarrow \cE(C_{\bG^*}(s_0)^{F^*}, 1).
\end{equation}
We remark that Lusztig \cite{Lu88} and Digne and Michel \cite{DiMi90} extended the Jordan decomposition map to the case that $\bG$ does not necessarily have connected center.

When $\bG$ is a classical group, the property (\ref{inner}) completely characterizes the bijection between $\cE(\bG, s)$ and $\cE(C_{\bG^*}(s)^{(\dot{w}_1 F)^*}, 1)$, but this is not true in general.  Digne and Michel \cite{DiMi90} provided a list of properties of such a bijection which does uniquely characterize it, which we state below.  We must recall some notions before doing so.

First, for any Levi subgroup $\bL^{\dot{w}F}$ of $\bG^{F}$, there is the {\it Lusztig induction} functor $R_{\bL^{\dot{w}F}}^{\bG^F}$, which takes characters of $\bL^{\dot{w}F}$ to generalized characters of $\bG^F$, which is parabolic induction in the case that the Levi is contained in an $F$-stable parabolic.  Technically, the Lusztig induction functor may depend on the choice of parabolic subgroup which contains the Levi, although this is now known to be independent in most cases \cite{BoMi11}.

The Deligne-Lusztig characters $R_{\bT^{wF}}^{\bG^F}(\theta)$ are constructed using certain $l$-adic cohomology groups with compact support, and there are eigenvalues of $F^d$ associated to unipotent characters $\chi \in \cE(\bG^F, 1)$ arising from the action of $F^d$ on the representation spaces.  Finally, given an isogeny $\varphi: (\bG, F) \rightarrow (\bG_1, F_1)$, there exists some dual isogeny $\varphi^*: (\bG_1^*, F_1^*) \rightarrow (\bG^*, F^*)$.  We will not need the details of these constructions.

We may now state the following result of Digne and Michel \cite[Theorem 7.1]{DiMi90}.  

\begin{thm}[Digne and Michel, 1990] \label{UniqueJord}  Given any $s \in \bT^*$ such that $W_F(s)$ is nonempty, there exists a unique bijection:
$$ J_s : \cE(\bG^F, s) \longrightarrow \cE(C_{\bG^*}(s)^{(\dot{w}_1 F)^*}, 1)$$
which satisfies the following conditions:
\begin{enumerate}
\item For any $\chi \in \cE(\bG^F, s)$, and any $w \in W_F(s)$, 
$$ \langle \chi, R_{\bT^{wF}}^{\bG^F}(\hat{s}) \rangle = \langle J_s(\chi), (-1)^{\ell(w_1)} R_{\bT^{*(wF)^*}}^{C_{\bG^*}(s)^{(\dot{w}_1 F)^*}} ({\bf 1}) \rangle.$$
\item If $s=1$ then:
\begin{itemize}
\item[(a)] The eigenvalues of $F^{d}$ associated to $\chi$ are equal, up to a power of $q^{d/2}$, to the eigenvalues of $F^{*d}$ associated to $J_1(\chi)$.
\item[(b)] If $\chi$ is in the principal series then $J_1(\chi)$ and $\chi$ correspond to the same character of the Hecke algebra.
\end{itemize}

\item If $z \in Z(\bG^{*F^*})$ is central, and $\chi \in \cE(\bG^F, s)$, then $J_{sz}(\chi \otimes \hat{z}) = J_s(\chi)$.
\item If $\bL$ is a standard Levi subgroup of $\bG$ such that $\bL^*$ contains $C_{\bG^*}(s)$ and such that $\bL$ is $\dot{w}F$-stable, then the following diagram is commutative:

$$\begin{CD}
\cE(\bG^F, s)           @>J_s>> \cE(C_{\bG^*}(s)^{(\dot{w}_1 F)^*}, 1) \\
@AAR_{\bL^{\dot{w}F}}^{\bG^F}A            @|\\
\cE(\bL^{\dot{w}F}, s)         @>J_s>> \cE(C_{\bL^*}(s)^{(\dot{v}\dot{w}F)^*}, 1)
\end{CD}$$
where $\dot{v} \dot{w} = \dot{w}_1$, and we extend $J_s$ by linearity to generalized characters.

\item Assume $(W, F)$ is irreducible, $(\bG, F)$ is of type $E_8$, and $(C_{\bG^*}(s), (\dot{w}_1 F)^*)$ is of type $E_7 \times A_1$ (respectively, $E_6 \times A_2$, respectively ${^2 E_6} \times {^2 A_2}$).  Let $\bL$ be a Levi of $\bG$ of type $E_7$ (respectively $E_6$, respectively $E_6$) which contains the corresponding component of $C_{\bG^*}(s)$.  Then the following diagram is commutative:
$$\begin{CD}
\cE(\bG^F, s)                  @>J_s>>     \cE(C_{\bG^*}(s)^{(\dot{w}_1 F)^*}, 1) \\
@AAR_{\bL^{\dot{w}_2F}}^{\bG^F}A                  @AAR_{\bL^{*(\dot{w}_2 F)^*}}^{C_{\bG^*}(s)^{(\dot{w}_1 F)^*}}A \\
\cE(\bL^{\dot{w}_2F}, s)^{\bullet}  @>J_s>>     \cE(\bL^{*(\dot{w}_2F)^*}, 1)^{\bullet}
\end{CD}$$
where the superscript $\bullet$ denotes the cuspidal part of the Lusztig series, and $w_2 = 1$ (respectively $1$, respectively the $W_{\bL}$-reduced element of $W_F(s)$ which is in a parabolic subgroup of type $E_7$ of $W$).

\item Given an epimorphism $\varphi: (\bG, F) \rightarrow (\bG_1, F_1)$ such that $\mathrm{ker}(\varphi)$ is a central torus, and semisimple elements $s_1 \in \bG_1^{*F_1^*}$, $s = \varphi^*(s_1) \in \bG^{*F^*}$, the following diagram is commutative:
$$\begin{CD}
\cE(\bG^F, s)                  @>J_s>>     \cE(C_{\bG^*}(s)^{(\varphi(\dot{w}_1) F)^*}, 1) \\
@AA{^t \varphi}A                           @VV{^t \varphi^*}V \\
\cE(\bG_1^{F_1}, s_1)           @>J_{s_1}>>  \cE(C_{\bG_1^*}(s_1)^{(\dot{w}_1 F_1)^*}, 1)
\end{CD}$$

\item If $\bG$ is a direct product, $\bG = \prod_i \bG_i$, then $J_{\prod_i s_i} = \prod_i J_{s_i}$.

\end{enumerate}
\end{thm}

We will need the following result, which follows from the above.

\begin{lemma} \label{ConjJ}  Let $s \in \bT^{*(w_1 F)^*}$, so that $s, s^{-1} \in C_{\bG^*}(s)^{(\dot{w}_1 F)^*}$, and suppose that there exists $\dot{v} \in \Nr_{\bG^*}(\bT^*)^{(\dot{w}_1 F)^*}$ such that $\dot{v} s \dot{v}^{-1} = s^{-1}$.  Let $J_s$, $J_{s^{-1}}$ be the maps as described in Theorem \ref{UniqueJord}.  If $J_s(\chi) = \psi$, then $J_{s^{-1}}(\chi) = {^{\dot{v}} \psi}$.
\end{lemma}
\begin{proof}  Let $s_0 \in \bG^{*F^*}$ be the semisimple element associated with $s$.  From the assumption, we have some $\dot{v}_0 \in \Nr_{\bG^*}(\bT^*)^{F^*}$ such that $\dot{v}_0 s_0 \dot{v}_0^{-1} =s_0^{-1}$.  Since $s_0$ and $s_0^{-1}$ are conjugate in $\bG^{*F^*}$, then $\cE(\bG^F, s) = \cE(\bG^F, s^{-1}) = \cE(\bG^F, {^{\dot{v}} s})$.  Given the element $\dot{v}_0 \in \bG^{*F^*}$, where $v_0 \in W^{*F^*}$, consider $w_0 \in W^{ F}$ which satisfies $w_0^* = v_0$, and take $\dot{w}_0 \in \Nr_{\bG}(\bT)^{F}$.  Then the inner automorphism given by $\mathrm{ad} \, \dot{w}_0 = \varphi$ on $\bG^{F}$, which acts trivially on $\cE(\bG^F, s) = \cE(\bG^F, {^{\dot{v}} s})$, is such that $\varphi^*$ acts on $C_{\bG^*}(s)^{(\dot{w}_1 F)^*}$ by $\mathrm{ad} \, \dot{v}$.
  By Theorem \ref{UniqueJord}, property (6), we then have the commutative diagram
$$\begin{CD}
\cE(\bG^F, s)                  @>J_s>>     \cE(C_{\bG^*}(s)^{(\dot{w}_1 F)^*}, 1) \\
@|                         @VV\mathrm{ad} \, \dot{v}V \\
\cE(\bG^F, {^{\dot{v}} s})           @>J_{^{\dot{v}} s}>>  \cE(C_{\bG^*}(s)^{(\dot{w}_1 F)^*}, 1)
\end{CD}$$
That is, if $J_s(\chi) = \psi$, then we have $J_{^{\dot{v}} s} (\chi) = {^{\dot{v}} \psi}$. 
\end{proof}

We note that the result of Digne and Michel, Theorem \ref{UniqueJord}, is studied in more detail in several places.  Enguehard \cite{En08} studies the implications of the result for the block theory of characters of finite reductive groups, and Cabanes and Sp\"ath \cite{CaSp13} study in detail and apply the equivariance property (6) of Theorem \ref{UniqueJord}.

\section{Proof of the main theorem} \label{Main}

The following result is \cite[Proposition 11.4]{dmbook}.  We will apply it repeatedly in the proof of the main result.

\begin{lemma} \label{DLbar} For any Levi subgroup $\bL^{\dot{w}F}$ of $\bG^F$, and any character $\alpha$ of $\bL^{\dot{w}F}$, we have $\overline{R_{\bL^{\dot{w}F}}^{\bG^F}(\alpha)} = R_{\bL^{\dot{w}F}}^{\bG^F} (\bar{\alpha})$.
\end{lemma}

We also need the following observation.

\begin{lemma} \label{Ebar} We have $\chi \in \cE(\bG^F, s)$ if and only if $\bar{\chi} \in \cE(\bG^F, s^{-1})$.  Thus, if $\chi$ is real-valued, $\cE(\bG^F, s) = \cE(\bG^F, s^{-1})$, and $s$ and $s^{-1}$ are conjugate in $\bG^{*(\dot{w}_1F)^*}$.
\end{lemma}
\begin{proof} We have $\chi \in \cE(\bG^F, s)$ if and only if  
$$\langle \chi, R_{\bT^{wF}}^{\bG^F}(\hat{s}) \rangle  = \langle \bar{\chi}, \overline{R_{\bT^{wF}}^{\bG^F}(\hat{s})} \rangle  \neq 0,$$
for some $w \in W_F(s)$.  From Lemma \ref{DLbar}, this holds if and only if 
$$\langle \bar{\chi}, R_{\bT^{wF}}^{\bG^F}(\widehat{s^{-1}}) \rangle \neq 0,$$
which holds precisely when $\bar{\chi} \in \cE(\bG^F, s^{-1})$.

Hence, if $\chi$ is real-valued, we have $\chi \in \cE(\bG^F, s)$ and $\bar{\chi} = \chi \in \cE(\bG^F, s^{-1})$.  Since Lusztig series which intersect must be equal, we have $\cE(\bG^F, s) = \cE(\bG^F, s^{-1})$.  This holds if and only if $s$ and $s^{-1}$ are conjugate in $\bG^{*(\dot{w}_1F)^*}$, that is, when their associated semisimple elements of $\bG^{*F^*}$ are conjugate in $\bG^{*F^*}$.
\end{proof}

We now arrive at our main result, which we state in terms of the notation in (\ref{JD2}) for convenience.

\begin{thm} \label{MainThm}  Suppose $\chi$ is an irreducible character of $\bG^F$, and $s_0 \in \bG^{*F^*}$ is semisimple, where $\chi \in \cE(\bG^F, s_0)$.  Let $J_{s_0}(\chi) = \nu$, where $J_{s_0}: \cE(\bG^F, s_0) \longrightarrow \cE(C_{\bG^*}(s_0)^{F^*}, 1)$ is the Jordan decomposition map.  Then $\chi$ is real-valued if and only if the following hold:
\begin{enumerate}
\item[(i)] The element $s_0$ is real in $\bG^{*F^*}$, so $h_0s_0 h_0^{-1} = s_0^{-1}$ for some $h_0 \in \bG^{*F^*}$.
\item[(ii)] If $h_0 \in \bG^{*F^*}$ satisfies $h_0 s_0 h_0^{-1} = s_0^{-1}$, then $\nu$ satisfies ${^{h_0} \nu} = \bar{\nu}$.
\end{enumerate}
That is, if $\chi$ corresponds to the pair $(s_0, \nu)$ in the Jordan decomposition, then $\chi$ is real-valued if and only if the pair $(s_0, \nu)$ is $\bG^{*F^*}$-conjugate to the pair $(s_0^{-1}, \bar{\nu})$.
\end{thm}
\begin{proof}   Since we apply Theorem \ref{UniqueJord}, we translate the statement into that notation.  If $s \in \bT^{*}$ is such that $W_F(s)$ is nonempty, suppose $\chi \in \cE(\bG^F, s)$, and let $J_s(\chi) = \psi \in \cE(C_{\bG^*}(s)^{(\dot{w}_1 F)^*}, 1)$.  Then we must show that $\chi = \bar{\chi}$ if and only if there is some $h \in \bG^{*(\dot{w}_1 F)^*}$ such that $h s h^{-1} = s^{-1}$ and ${^h \psi} = \bar{\psi}$.

We first prove that if $J_s(\chi) = \psi$, then $J_{s^{-1}}(\bar{\chi}) = \bar{\psi}$.   By Lemma \ref{Ebar}, the claim makes sense.  We prove the claim by induction on the semisimple rank of $\bG$.  For the base case, we may assume that $\bG = \bT$ is a torus, in which case each Lusztig series is a singleton, and for any irreducible of $\bG^F$, $J_s(\chi) = {\bf 1}$, the trivial character.  So $J_{s^{-1}}(\bar{\chi}) = {\bf 1}$ as well.  

We now assume the claim holds for any group with smaller semisimple rank than $(\bG, F)$, and we prove the statement holds for $(\bG, F)$ (and any other group in the same isogeny class, so with the same semisimple rank).  We note that for every $s \in \bT^*$ with $W_F(s)$ nonempty, the map $\mu_{s^{-1}}(\bar{\chi}) = \bar{\psi}$ is well-defined since every element of $\cE(\bG, s^{-1})$ is of the form $\bar{\chi}$ for some $\chi \in \cE(\bG^F, s)$.  We show that this map satisfies the properties listed in Theorem \ref{UniqueJord}, and then by uniqueness we must have $\mu_{s^{-1}} = J_{s^{-1}}$.  We show this map satisfies the properties one by one, and we only need the induction hypothesis for some of the parts.

For property (1), by Lemma \ref{DLbar}, and the fact that $\langle \chi, R_{\bT^{wF}}^{\bG^F}(\hat{s}) \rangle$ is an integer, we have
$$\langle \chi, R_{\bT^{wF}}^{\bG^F}(\hat{s}) \rangle = \overline{ \langle \chi, R_{\bT^{wF}}^{\bG^F}(\hat{s}) \rangle } =  \langle \bar{\chi}, \overline{R_{\bT^{wF}}^{\bG^F}(\hat{s})} \rangle = \langle \bar{\chi}, R_{\bT^{wF}}^{\bG^F}(\widehat{s^{-1}}) \rangle,$$
and similarly,
\begin{align*}
\langle \psi, (-1)^{\ell(w_1)} R_{\bT^{*(wF)^*}}^{C_{\bG^*}(s)^{(\dot{w}_1 F)^*}} ({\bf 1}) \rangle & =  \overline{\langle \psi, (-1)^{\ell(w_1)} R_{\bT^{*(wF)^*}}^{C_{\bG^*}(s)^{(\dot{w}_1 F)^*}} ({\bf 1}) \rangle} \\
&  = \langle \bar{\psi}, (-1)^{\ell(w_1)} R_{\bT^{*(wF)^*}}^{C_{\bG^*}(s)^{(\dot{w}_1 F)^*}} ({\bf 1}) \rangle.
\end{align*}
Since we have 
$$\langle \chi, R_{\bT^{wF}}^{\bG^F}(\hat{s}) \rangle = \langle \psi, (-1)^{\ell(w_1)} R_{\bT^{*(wF)^*}}^{C_{\bG^*}(s)^{(\dot{w}_1 F)^*}} ({\bf 1}) \rangle,$$
then it follows we have
$$\langle \bar{\chi}, R_{\bT^{wF}}^{\bG^F}(\widehat{s^{-1}}) \rangle = \langle \bar{\psi}, (-1)^{\ell(w_1)} R_{\bT^{*(wF)^*}}^{C_{\bG^*}(s)^{(\dot{w}_1 F)^*}} ({\bf 1}) \rangle,$$
so that $\mu_{s^{-1}}$ satisfies (1) of Theorem \ref{UniqueJord}.

Property (2) of Theorem \ref{UniqueJord} concerns only the case $s=1$, and the bijection $J_1: \cE(\bG^F, 1) \rightarrow \cE(\bG^{*F^*}, 1)$ between unipotent characters of $\bG^F$ and $\bG^{*F^*}$.  Suppose $J_1(\chi) = \psi$, and that $\alpha, \alpha'$ are the eigenvalues of $F^{d}$ associated with $\chi$ and $\bar{\chi}$, and $\beta,\beta'$ are the eigenvalues of $F^{*d}$ associated with $\psi$ and $\bar{\psi}$, respectively.  By property (2a) of Theorem \ref{UniqueJord}, we know that $\alpha \beta^{-1}$ is a power of $q^{d/2}$.  By \cite[Proof of Corollary 3.9]{Lu78}, we know $\alpha \alpha'$ and $\beta \beta'$ are both powers of $q^{d}$.  Thus $\alpha'$ and $\beta'$ are equal up to a power of $q^{d/2}$.  Thus property (2a) of Theorem \ref{UniqueJord} holds for the map $\mu_1$.  For property (2b), we assume $\chi$ is in the principal series, meaning that it is a constituent of $\Ind_{\bB^F}^{\bG^F}(\mathbf{1})$, and so $\psi$ is also in the principal series, and they correspond to the same character $\zeta$, of the Hecke algebra $\mathcal{H}(\bG^F, \bB^F)$ (which is identified with $\mathcal{H}(\bG^{*F^*}, \bB^{*F^*})$ through the natural isomorphism $\delta: W \rightarrow W^*$ of Weyl groups).  If $\chi$ (and $\psi$) is in the principal series, then so is $\bar{\chi}$ (and $\bar{\psi}$).  Using \cite[Theorem 11.25(ii)]{CuRe81}, for example, we see that if $\chi$ corresponds to the character $\zeta$ of $\mathcal{H}(\bG^F, \bB^F)$, then $\bar{\chi}$ corresponds to $\bar{\zeta}$.  Thus if $\chi$ and $\psi$ correspond to the same character $\zeta$, then $\bar{\chi}$ and $\bar{\psi}$ both correspond to $\bar{\zeta}$, where $\mu_1(\bar{\chi}) = \bar{\psi}$.  So Property (2b) holds for $\mu_1$ as well.

For property (3), given any $z \in Z(\bG^{*F^*})$, we must show $\mu_{s^{-1}}(\bar{\chi}) = \mu_{s^{-1} z}( \bar{\chi} \otimes \hat{z})$.  If $J_s(\chi) = \psi$, then $J_{sz^{-1}} (\chi \otimes \widehat{z^{-1}}) = \psi$, and $\mu_{s^{-1}}(\bar{\chi}) = \bar{\psi}$.  We then have
$$ \mu_{s^{-1} z}(\bar{\chi} \otimes \hat{z}) = \overline{J_{sz^{-1}}( \chi \otimes \widehat{z^{-1}})} = \bar{\psi},$$
as required.

For property (4), since $\bL$ has smaller semisimple rank than $\bG$, we may use the induction hypothesis.  That is, for any $\xi \in \cE(\bL^{\dot{w}F}, s)$, if $J_s(\xi) = \psi$, then $J_{s^{-1}}(\bar{\xi}) = \bar{\psi}$.  Also, if $R_{\bL^{\dot{w}F}}^{\bG^F}(\xi) = \chi$, then $\overline{R_{\bL^{\dot{w}F}}^{\bG^F}(\xi)} = \bar{\chi}$.  It follows that the diagram in property (4) of Theorem \ref{UniqueJord} commutes when we replace $J_s$ with $\mu_{s^{-1}}$ in the top row and $J_s$ with $J_{s^{-1}}$ in the bottom row, as desired.

The proof that property (5) is satisfied by $\mu_{s^{-1}}$ is similar to (4).  Since $\bL$ has smaller semisimple rank than $\bG$, then by the induction hypothesis if $\xi \in \cE(\bL^{\dot{w}_2 F}, s)^{\bullet}$ and $J_s(\xi) = \lambda \in \cE(\bL^{*(\dot{w}_2F)^*}, 1)^{\bullet}$, then $J_{s^{-1}}(\bar{\xi}) = \bar{\lambda}$.  Also, if 
$$ R_{\bL^{*(\dot{w}_2 F)^*}}^{C_{\bG^*}(s)^{(\dot{w}_1 F)^*}}(\lambda) = \psi, \quad \text{ then } \quad R_{\bL^{*(\dot{w}_2 F)^*}}^{C_{\bG^*}(s)^{(\dot{w}_1 F)^*}}(\bar{\lambda}) = \bar{\psi}.$$
 If $R_{\bL^{\dot{w}_2 F}}^{\bG^F}(\xi) = \chi$, then $R_{\bL^{\dot{w}_2 F}}^{\bG^F}(\bar{\xi}) = \bar{\chi}$.  Since the diagram from property (5) of Theorem \ref{UniqueJord} commutes for $J_s$, we have $J_s(\chi) = \psi$, so $\mu_{s^{-1}}(\bar{\chi}) = \bar{\psi}$.  It follows that the diagram commutes using $J_{s^{-1}}$ on the bottom row and $\mu_{s^{-1}}$ on the top row, so that property (5) holds for $\mu_{s^{-1}}$.

For property (6), let $\varphi: (\bG, F) \rightarrow (\bG_1, F_1)$ be an epimorphism with kernel a central torus, i.e. an isogeny.  Then $\varphi^*: (\bG_1^*, F_1^*) \rightarrow (\bG^*, F^*)$ is injective.  Suppose $\chi_1 \in \cE(\bG_1^{F_1}, s_1)$, ${^t \varphi}(\chi_1) = \chi \in \cE(\bG^F, s)$, with $J_s(\chi) = \psi$ and $J_{s_1}(\chi_1) = \psi_1$.  Since property (6) holds for $J_s$ and $J_{s_1}$, then we have ${^t \varphi^*}(\psi) = \psi_1$.  For any $x$, we have 
$$ \psi_1(x^{-1}) = ({^t \varphi^*}(\psi))(x^{-1}) =  \psi(\varphi^*(x)^{-1}) = \overline{\psi(\varphi^*(x))}.$$
Thus, ${^t \varphi^*}(\bar{\psi}) = \bar{\psi}_1$.  We also have $\varphi(s^{-1}) = s_1^{-1}$, and ${^t \varphi}(\bar{\chi}_1) = \bar{\chi}$.  Now, the diagram from property (6) is commutative when we have $\mu_{s_1^{-1}}$ in the bottom row and $\mu_{s^{-1}}$ in the top row, as desired.

Finally, property (7) follows quickly, since if $\chi = \prod_i \chi_i$, and $J_s(\chi) = \psi = \prod_i \psi_i$, where $s = \prod_i s_i$ and $J_{s_i}(\chi_i) = \psi_i$, then 
$$\mu_{\prod_i s_i^{-1}} (\bar{\chi}) = \mu_{s^{-1}}(\bar{\chi}) = \bar{\psi} = \prod_i \bar{\psi_i}  = \prod_i \mu_{s_i^{-1}}(\bar{\chi_i}).$$
We now have the claim that if $J_s(\chi) = \psi$, then $J_{s^{-1}}(\bar{\chi}) = \bar{\psi}$.

Suppose $s \in \bT^*$ such that $W_F(s)$ is nonempty.  Using the notation of Lemma \ref{NormFrob}, suppose $h \in \bG^{*(\dot{w}_1F)^*}$ normalizes $C_{\bG^*}(s)^{(\dot{w}_1 F)^*}$.  By \cite[(1.27)]{BrMaMi93} it follows that  the automorphism given by conjugation by $h$ on $C_{\bG^*}(s)^{(\dot{w}_1 F)^*}$ permutes the set of unipotent characters $\cE(C_{\bG^*}(s)^{(\dot{w}_1 F)^*}, 1)$.  If $h s h^{-1} = s^{-1}$ for $h \in \bG^{*(\dot{w}_1F)^*}$, then $h$ normalizes $C_{\bG^*}(s)^{(\dot{w}_1 F)^*}$, and any other element $h_1 \in \bG^{*(\dot{w}_1F)^*}$ with the property $h_1 s h_1^{-1} = s^{-1}$ satisfies $h_1 \in h C_{\bG^{*(\dot{w}_1F)^*}}(s)$.  In particular, this means the conjugation action of $h$ on $\cE(C_{\bG^*}(s)^{(\dot{w}_1 F)^*}, 1)$ is independent of the choice of $h$ with the property $h s h^{-1} = s^{-1}$.

Suppose $\chi$ is real-valued, with $J_s(\chi) = \psi$.  Then $\chi = \bar{\chi}$, and from Lemma \ref{Ebar} it follows that $\cE(\bG^F, s) = \cE(\bG^F, s^{-1})$.  Then $hsh^{-1} = s^{-1}$ for some $h \in \bG^{*(\dot{w}_1 F)^*}$.  Since $s, s^{-1} \in \bT^{*(w_1 F)^*}$, then by Lemma \ref{NormFrob}, there is some $\dot{v} \in \Nr_{\bG^*}(\bT^*)^{(\dot{w}_1 F)^*}$ such that $\dot{v} s \dot{v}^{-1} = s^{-1}$.  By Lemma \ref{ConjJ}, we have $J_{^{\dot{v}} s}(\chi) =  {^{\dot{v}} \psi}$.  By the claim proved above, we also have $J_{s^{-1}}(\chi) = \bar{\psi}$.  Since ${^{\dot{v}} s} = s^{-1}$, then we have ${^{\dot{v}} \psi} = \bar{\psi}$, and then also ${^h \psi} = \bar{\psi}$.

Conversely, suppose $\chi \in \cE(\bG^F, s)$, where $hsh^{-1} = s^{-1}$ for some $h \in \bG^{*(\dot{w}_1 F)^*}$, $J_s(\chi) = \psi$ and ${^h \psi} = \bar{\psi}$.  Again by Lemma \ref{NormFrob}, we have ${^{\dot{v}} s} = s^{-1}$ for some $\dot{v} \in \Nr_{\bG^*}(\bT^*)^{(\dot{w}_1 F)^*}$, and ${^{\dot{v}} \psi} = \bar{\psi}$.  We then have $J_{s^{-1}}(\bar{\chi}) = \bar{\psi}$ as proved above, and we have $J_{^{\dot{v}} s}(\chi) = {^{\dot{v}} \psi}$ by Lemma \ref{ConjJ}.  Since $\bar{\psi} = {^{\dot{v}} \psi}$ and $s^{-1} = {^{\dot{v}} s}$, we have $J_{s^{-1}}(\bar{\chi}) = J_{s^{-1}}(\chi)$, and so $\chi = \bar{\chi}$ since $J_{s^{-1}}$ is a bijection.  Thus $\chi$ is real-valued.
\end{proof}

\noindent {\bf Remarks. }  It has been pointed out to the authors by P.H. Tiep that Theorem \ref{MainThm} does not hold if the assumption is dropped that the center of $\bG$ is connected.  It fails in the case that $\bG^F = \mathrm{SL}(2, \FF_q)$ for those characters in the Lusztig series corresponding to $s \in \bG^*$ with $C_{\bG^*}(s)$ disconnected.  However, it is possible that the result may hold with the milder assumption that $C_{\bG^*}(s)$ is connected.  This statement for the semisimple characters of $\bG^F$ is implied by a lemma of Navarro and Tiep \cite[Lemma 9.1]{NaTi08}.

\end{document}